\documentclass{article}

\usepackage[utf8]{inputenc}
\usepackage[english]{babel}

\usepackage{geometry}
\usepackage{fancyhdr}
\usepackage{titlesec}
\usepackage{indentfirst}
\usepackage{float}

\usepackage{tikz}
\usetikzlibrary{arrows.meta, positioning}

\usepackage{amsmath,comment}
\usepackage{amssymb}
\usepackage{mathrsfs}
\usepackage{amsthm}
\usepackage{amsfonts}
\usepackage{ifthen}

\usepackage{aliascnt}

\theoremstyle{plain}

\newtheorem{thm}{Theorem}[section]

\newaliascnt{prop}{thm}
\newtheorem{prop}[prop]{Proposition}
\aliascntresetthe{prop}

\newaliascnt{lem}{thm}
\newtheorem{lem}[lem]{Lemma}
\aliascntresetthe{lem}

\newaliascnt{cor}{thm}
\newtheorem{cor}[cor]{Corollary}
\aliascntresetthe{cor}

\newaliascnt{defn}{thm}
\newtheorem{defn}[defn]{Definition}
\aliascntresetthe{defn}

\newaliascnt{prob}{thm}

\aliascntresetthe{prob}

\newtheorem{introthm}{Theorem}

\newaliascnt{introcor}{introthm}

\aliascntresetthe{introcor}

\newtheorem*{thm*}{Theorem}
\newtheorem*{cor*}{Corollary}

\theoremstyle{remark}

\newaliascnt{example}{thm}

\aliascntresetthe{example}

\newtheorem{remark}{Remark}

\newtheorem*{remark*}{Remark}

\usepackage[linktoc=all]{hyperref}
\usepackage{cleveref}

\crefname{thm}{Theorem}{Theorems}
\Crefname{thm}{Theorem}{Theorems}

\crefname{prop}{Proposition}{Propositions}
\Crefname{prop}{Proposition}{Propositions}

\crefname{lem}{Lemma}{Lemmas}
\Crefname{lem}{Lemma}{Lemmas}

\crefname{cor}{Corollary}{Corollaries}
\Crefname{cor}{Corollary}{Corollaries}

\crefname{defn}{Definition}{Definitions}
\Crefname{defn}{Definition}{Definitions}

\crefname{prob}{Problem}{Problems}
\Crefname{prob}{Problem}{Problems}

\crefname{introthm}{Theorem}{Theorems}
\Crefname{introthm}{Theorem}{Theorems}

\crefname{introcor}{Corollary}{Corollaries}
\Crefname{introcor}{Corollary}{Corollaries}

\crefname{example}{Example}{Examples}
\Crefname{example}{Example}{Examples}

\crefname{remark}{Remark}{Remarks}
\Crefname{remark}{Remark}{Remarks}

\crefname{enumi}{Item}{Items}
\Crefname{enumi}{Item}{Items}

\newcommand{\bbE}{\mathbb{E}}
\newcommand{\bbN}{\mathbb{N}}
\newcommand{\bbP}{\mathbb{P}}

\newcommand{\bbZ}{\mathbb{Z}}

\newcommand{\cP}{\mathcal{P}}

\newcommand{\rar}{\rightarrow}

\newcommand{\ol}[1]{\overline{#1}}

\newcommand{\abs}[1]{\left|#1\right|}

\newcommand{\divides}{\bigm|}

\renewcommand{\Pr}{\mathbb{P}}

\title{The isomorphism problem for random generalized Baumslag-Solitar groups with many edges}

\author{
Dario Ascari, Alessandra Caraceni
}

\date{}

\begin{document}

\maketitle

\begin{abstract}
    The isomorphism problem for generalized Baumslag-Solitar (GBS) groups is a long-standing open problem. We prove that, in the random model considered here, a GBS graph with many edges is flexible with high probability. Since the isomorphism problem is decidable for flexible GBS graphs, this gives a high-probability decidability result for random GBS groups in this regime.
\end{abstract}



\section{Introduction}

The isomorphism problem for generalized Baumslag-Solitar groups (GBS groups) is a long-standing open question. Its relevance comes from the study of the JSJ decomposition, which is a key tool both in the classification of (closed orientable) $3$-manifolds, and in the study of hyperbolic groups and their automorphism groups \cite{Sel01,DG11,GL17}. In fact, solving the isomorphism problem for GBS groups would also facilitate the classification of large families of graphs of groups \cite{ACK-iso1}. Over the years, many attempts have been made, yielding solutions in several restricted cases \cite{For02,For03,For06,Lev07,CF08,Dud17,Wan25,ACK-iso2,ACK-iso3}. However, the problem turns out to be highly challenging, and it remains open in the general setting. In this paper, using the results of \cite{ACK-iso2}, we show that the isomorphism problem is decidable with high probability for \emph{random GBS groups with many edges}, in a sense we will clarify shortly.

\paragraph{Description of the random model.} 
A GBS group is defined as the fundamental group of a graph of groups with vertex groups and edge groups isomorphic to $\bbZ$. Every GBS group can be represented by a finite connected graph, where every end of every edge is equipped with a label, which is a non-zero integer (see \Cref{sec:preliminaries} for the details). Thus, in order to provide a GBS group, one can specify a finite connected graph, and, letting $E$ be the cardinality of its edge set, a sequence of $2E$ non-zero integers to be attached to the ends of the graph's edges.

We use the following model in order to sample a \emph{random} GBS group. Having chosen positive integers $V, E, D, N$ and a set of prime numbers $\cP$ of cardinality $\abs{\cP}=D$, we consider the following object, which we call \emph{a random GBS graph with parameters $V,E,D,N$}. First of all, we fix the set $[V]=\{1,2,\ldots,V\}$ to be the vertex set; we then sample a list of $E$ edges, each of which joins two endpoints chosen independently and uniformly at random in $[V]$ (note that the two endpoints of an edge may coincide). We then sample $2E$ independent and identically distributed labels, one at each end of each edge, of the form $(-1)^\epsilon\cdot\prod_{p_i\in\cP}p_i^{\alpha_i}$, where $\alpha_i\in\{0,\dots,N-1\}$ and $\epsilon\in\{0,1\}$ are chosen uniformly at random and independently. For convenience in the random setup, we allow the GBS graph to be disconnected in general; if this is the case, no ensuing GBS group will be defined. If the graph is connected -- which, in the regime of random GBS graphs ``with many edges'', will be the case with high probability, the sampled random GBS graph yields a corresponding GBS group.

The above defines a natural probabilistic model for random GBS groups. 
The random graph is sampled in a very simple way reminiscent of the Erdős–Rényi model, designed to allow for loops and multiple edges (both features that do occur in general for GBS graphs). As for the labels, we fix a set of primes $\cP$, we impose that the labels have factorizations using only primes from $\cP$, and we choose the exponents in the factorizations independently and uniformly at random in $[0,N-1]\cap\mathbb{Z}$. The set of primes $\cP$ is easily seen to be an isomorphism invariant of the GBS group; moreover, when trying to decide the isomorphism problem between two GBS groups, the set of primes $\cP$ is completely irrelevant, and the isomorphism problem is decided only by the exponents appearing in the factorizations \cite{ACK-iso1}. For this reason, it seems reasonable to us to fix the set of primes $\cP$ -- or indeed, just the cardinality $D=\abs{\cP}$ of this set of primes -- and to choose the exponents independently at random.

\begin{remark}
    A different natural choice for a model of random GBS is to pick the labels as non-zero integers in $[-N,N]$ independently and uniformly at random. 
    However, the complexity of the isomorphism problem arises from divisibility relations between labels in the graph (compare with the notion of algebraically rigid GBS group in \cite{Lev07}). We discuss this alternative model, and we show that random GBS graphs obtained in this way are rigid with high probability (see \Cref{introthm:random-GBS-is-rigid} below), and thus the isomorphism problem can be decided algorithmically \cite{CF08}.
\end{remark}

\paragraph{Random GBS graphs are flexible.} 
We introduce the notion of \emph{flexible GBS graph} (see \Cref{def:flexible-GBS}, and more generally \Cref{sec:flexibility} for more about the notion of flexibility), and we prove that, using the algorithm described in \cite{ACK-iso2}, the isomorphism problem is decidable within the class of flexible GBS graphs.

\begin{introthm}[\cite{ACK-iso2}]\label{introthm:isomorphism-for-flexible}
    There is an algorithm that, given two flexible GBS graphs, determines whether or not they represent isomorphic GBS groups.
\end{introthm}

We now estimate the probability that a random GBS graph is flexible. 

Consider the random GBS graph with $V$ vertices, $E$ edges, $D$ primes and exponents up to $N$. Note that, under the mild assumption of being fully reduced, the values $V,E,D$ are isomorphism invariant of the GBS group, so it makes sense to compare GBS groups while keeping these parameters fixed. On the other hand, the parameter $N$ describing the size of the exponents is not an isomorphism invariant: this motivates the following \Cref{def:probability-limit}.

\begin{defn}\label{def:probability-limit}\
    \begin{enumerate}
        \item For positive integers $V,E,D,N$, let $p(V,E,D,N)\in[0,1]$ be the probability that a random GBS graph with those parameters is flexible.
        \item For positive integers $V,E,D$, define the \emph{probability that the random GBS graph with $V$ vertices, $E$ edges and $D$ primes is flexible} as
    \[
        P(V,E,D):=\liminf_{N\rar +\infty}p(V,E,D,N).
    \]
    \end{enumerate}
\end{defn}

The main result of this paper is \Cref{introthm:random-is-flexible}, stating that a random GBS is flexible with high probability, provided it has sufficiently many edges.

\begin{introthm}[\Cref{thm:random-GBS}]\label{introthm:random-is-flexible}
    Let $0<\epsilon<\frac{1}{50}$. Suppose that $E\ge V^{2+\epsilon}$ and $E\ge e^{2D/\epsilon^2}$. Then
    \[
        P(V,E,D)\ge 1-\frac{1}{E^{\epsilon/8}}.
    \]
\end{introthm}

In particular, for all (positive integer) sequences $V_n,E_n,D_n$ with $E_n\ge V_n^{2+\epsilon}$ and $E_n\ge e^{2D_n/\epsilon^2}$, provided that $\lim_{n\to\infty}E_n=\infty$, we have that $\lim_{n\to+\infty} P(V_n,E_n,D_n)=1$. In other words, a random GBS graph with many edges is flexible with high probability, allowing us to decide the isomorphism problem.

We observe that the complete graph has $E\approx\frac12 V^2$. Therefore, the condition $E\ge V^{2+\epsilon}$ can be viewed as requiring slightly more topological edges than a complete graph on $V$ vertices, by a polynomial factor of $V^\epsilon$. The condition $E\ge e^{2D/\epsilon^2}$ (or some other condition stating that $E$ is big enough compared to $D$) is also necessary. In fact, if we fix $V,E$ and we let $D\rar+\infty$, then the probability that the ensuing random GBS graph is rigid tends to $1$.

The first technical ingredient in the proof of \Cref{introthm:random-is-flexible} is an estimate for the number of labels which are \emph{minimal} in the sense of divisibility among those attached near the same vertex. We show that, choosing $k$ labels at random (using our model), the number of minimal labels is approximately $\frac{(\log k)^{D-1}}{(D-1)!}$. One can then establish that, if fewer than $\sqrt{E}$ of the $2E$ labels are minimal, then with high probability no edge has minimal labels at both ends (see \Cref{lem:no-pair-of-minimals}). For GBS graphs where no edge carries two minimal labels, flexibility follows from \cite{ACK-iso2}.

Finally, in \cite{ACK-iso1} the isomorphism problem is essentially reduced to the case of one-vertex GBS groups. We give a variant of \Cref{introthm:random-is-flexible} for random GBS graphs with one vertex in \Cref{introthm:random-onevertex-is-flexible}.

\begin{introthm}[\Cref{thm:random-GBS-one-vertex}]\label{introthm:random-onevertex-is-flexible}
    Suppose that $E\ge e^{9D}$. Then
    \[
        P(1,E,D)\ge 1-\frac{1}{\sqrt[4]{E}}.
    \]
\end{introthm}

\paragraph{An alternative random model.} 
Another natural model is obtained by fixing $V, E$, and $L$, sampling the graph structure as above, and choosing the $2E$ labels independently and uniformly from $\{\pm1,\pm2,\ldots,\pm L\}$. In this model, if $L$ is large enough, divisibility relations between labels at a common vertex are unlikely: for two random integers of size at most $L$, the probability that one divides the other is small. Consequently, this model produces rigid GBS graphs with high probability. We prove the following:

\begin{introthm}[\Cref{thm:random-GBS-is-rigid}]\label{introthm:random-GBS-is-rigid}
    Consider a random GBS graph with $V$ vertices, $E$ edges and labels at most $L$, conditioned on being connected. Suppose that $L\ge (2E)^4$. Then the GBS graph is rigid with probability at least $1-\frac{1}{\sqrt[4]{L}}$.
\end{introthm}

Every rigid GBS has (essentially) only one labeled graph representing it. Therefore, the isomorphism problem is decidable for GBSs with this property.

\subsection*{Acknowledgments} Ascari was funded by the Basque Government grant IT1483-22 and by the Juan de la Cierva fellowship JDC2024-054855-I of the Spanish Government.

\section{Preliminaries}\label{sec:preliminaries}

\subsection{Graphs}

We consider graphs as combinatorial objects, following the notation of \cite{Ser77}. A \emph{graph} is a quadruple $\Gamma=(V(\Gamma),E(\Gamma),\ol{\cdot},\iota)$ consisting of a set $V(\Gamma)$ of \emph{vertices}, a set $E(\Gamma)$ of \emph{edges}, a map $\ol{\cdot}:E(\Gamma)\rar E(\Gamma)$ called \emph{reverse} and a map $\iota:E(\Gamma)\rar V(\Gamma)$ called \emph{initial vertex}; we require that, for every edge $e\in E(\Gamma)$, we have $\ol{e}\not=e$ and $\ol{\ol{e}}=e$. For an edge $e\in E(\Gamma)$, we denote by $\tau(e)=\iota(\ol{e})$ the \emph{terminal vertex} of $e$. The terminal and initial vertices of an edge $e$ will often be referred to as \emph{endpoints} of $e$.

It is possible to define the topological realization of a graph $\Gamma$, which is a CW complex. The number of $0$-cells is $\abs{V(\Gamma)}$ and the number of $1$-cells is $\frac{1}{2}\abs{E(\Gamma)}$ (because for every topological edge, the set $E(\Gamma)$ contains the edge and its reverse).

\subsection{Generalized Baumslag-Solitar groups}

A GBS group is the fundamental group of a graph of groups with vertex groups and edge groups isomorphic to $\bbZ$ (see \cite{Ser77} for the notion of graph of groups). This means that every vertex $u$ in the graph has attached a group $G_u\cong\bbZ$, and similarly every edge $e$ in the graph has attached a group $G_e\cong\bbZ$; and whenever $\tau(e)=u$ we have an injective homomorphism $G_e\rar G_u$. Since every injective homomorphism $\bbZ\rar\bbZ$ is given by multiplication by a non-zero integer, a GBS group can be represented as a finite graph where every topological edge is equipped with two labels (one at each end of the edge), which are the non-zero integers giving the injective homomorphisms. This motivates the following definition:

\begin{defn}
    A \emph{GBS graph} is a pair $(\Gamma,\psi)$, where $\Gamma$ is a finite graph and $\psi:E(\Gamma)\rar\bbZ\setminus\{0\}$ is a function.
\end{defn}

The reader should think of $\psi(e)$ as the label written on the edge $e$ near the terminal endpoint $\tau(e)$; an example of a GBS graph is shown in \Cref{fig:GBS-graph}.

A GBS graph determines a GBS group $G$, whose presentation is given as follows. Let $(\Gamma,\psi)$ be a GBS graph, and choose a maximal tree $T=\ol{T}\subseteq E(\Gamma)$. For every vertex $v\in V(\Gamma)$, we consider a generator $a_v$, and for every edge $e\in E(\Gamma)\setminus T$ we consider a generator $t_e$. The GBS group $G$ is given by the following finite presentation:

\vspace{0.15cm}

\noindent 
\begin{equation}\label{eq:presentation}
\begin{tabular}{l l l}
generators: 
& $a_v$ & $v\in V(\Gamma)$;\\[0.06cm]
& $t_e$ & $e\in E(\Gamma)\setminus T$;\\[0.12cm]
relations: 
& $a_v^\ell=a_u^m$ & $e\in T$ with $\iota(e)=v,\tau(e)=u,\psi(e)=m,\psi(\ol{e})=\ell$;\\[0.06cm]
& $t_{\ol{e}}=t_e^{-1}$ & $e\in E(\Gamma)\setminus T$;\\[0.06cm]
& $a_v^\ell t_e = t_e a_u^m$ & $e\in E(\Gamma)\setminus T$ with $\iota(e)=v,\tau(e)=u,\psi(e)=m,\psi(\ol{e})=\ell$.\\
\end{tabular}
\end{equation}

\vspace{0.15cm}

It is easy to see that different choices of the maximal tree $T$ will define isomorphic groups $G$. Every GBS group $G$ can be obtained from some GBS graph by using the above presentation. Different GBS graphs can produce the same GBS group $G$: deciding when two GBS graphs represent isomorphic GBS groups is known as the \emph{isomorphism problem} for GBS groups.

\begin{figure}[H]
\centering
\includegraphics[scale=1]{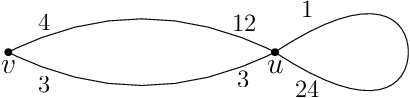}
\caption{The figure shows a GBS graph $(\Gamma,\psi)$ with two vertices $v,u$ and three edges $e_1,e_2,e_3$ (and their reverses). The edge $e_1$ goes from $v$ to $u$ and has $\psi(\ol{e}_1)=4$ and $\psi(e_1)=12$. The edge $e_2$ goes from $v$ to $u$ and has $\psi(\ol{e}_2)=\psi(e_2)=3$. The edge $e_3$ goes from $u$ to $u$ and has $\psi(\ol{e}_3)=1$ and $\psi(e_3)=24$.}
\label{fig:GBS-graph}
\end{figure}

\subsection{Fully reduced GBS graphs and set of primes}

\newcommand{\geqc}{\succeq_{\mathrm{qc}}}
\newcommand{\qcnj}{\sim_{\mathrm{qc}}}

Given a group $G$, consider two elements $g,g'\in G$: we write $g\geqc g'$ if $g$ is conjugate to a power of $g'$. 
Let $(\Gamma,\psi)$ be a GBS graph with corresponding GBS group $G$, and for a vertex $v\in V(\Gamma)$ we denote by $a_v$ the corresponding generator of $G$, as in presentation (\ref{eq:presentation}).

\begin{defn}
    We say that $(\Gamma,\psi)$ is \emph{fully reduced} if there are no distinct vertices $v,u\in V(\Gamma)$ such that $a_v\geqc a_u$.
\end{defn}

It is easy to see that the above definition does not depend on the maximal tree $T\subseteq E(\Gamma)$ used to define the presentation (\ref{eq:presentation}). In what follows, we will always work with fully reduced GBS graphs; this requirement is not restrictive at all, as for every GBS graph, there is a fully reduced one representing the same GBS group (see \cite{For06}). We also point out that, given a GBS graph, we can algorithmically decide whether it is fully reduced or not, and in case it is not, we can algorithmically compute a fully reduced GBS graph representing the same GBS group (see \cite{ACK-iso1}).

We will also use the following sufficient condition to ensure that our GBS graphs are fully reduced. For a random GBS, the probability that there is a label equal to $\pm1$ is very low, and thus the criterion of \Cref{lem:fully-reduced} will be enough for our purposes.

\begin{lem}\label{lem:fully-reduced}
    Let $(\Gamma,\psi)$ be a GBS graph. Suppose that for all $e\in E(\Gamma)$ we have $\psi(e)\not=\pm1$. Then $(\Gamma,\psi)$ is fully reduced.
\end{lem}
\begin{proof}
    Use the algorithm to compute quasi-conjugacy classes described in \cite{ACK-iso2}. When no label is $\pm1$, the algorithm terminates immediately, giving as output a trivial quasi-conjugacy class.
\end{proof}

\begin{defn}\label{def:set-of-primes}
    For a GBS graph $(\Gamma,\psi)$, define its \emph{set of primes}
    $$\cP(\Gamma,\psi):=\{r\in\bbN \text{ prime } : r\divides \psi(e) \text{ for some } e\in E(\Gamma)\}.$$
\end{defn}

\begin{prop}\label{prop:invariant-set-of-primes}
    Let $(\Gamma,\psi),(\Gamma',\psi')$ be two fully reduced GBS graphs representing isomorphic GBS groups. Then $\cP(\Gamma,\psi)=\cP(\Gamma',\psi')$.
\end{prop}
\begin{proof}
    See \cite{ACK-iso1}.
\end{proof}

Note that \Cref{prop:invariant-set-of-primes} is false if we do not assume that the GBS graphs are fully reduced.

\subsection{Flexible GBS graphs}\label{sec:flexibility}

\begin{defn}\label{def:minimal-edge}
    Let $(\Gamma,\psi)$ be a GBS graph. An edge $e\in E(\Gamma)$ is called \emph{minimal} if there is no $e'\in E(\Gamma)$ with $\tau(e')=\tau(e)$ and $\psi(e')\divides\psi(e)$ and $\psi(e')\not=\pm\psi(e)$.
\end{defn}

In other words, $e\in E(\Gamma)$ is minimal if the label at the terminal end of $e$ is not divided properly by any other adjacent label. For example, in the particular case where $e$ is a loop with $\psi(\ol{e})\divides\psi(e)$ and $\psi(\ol{e})\not=\pm\psi(e)$, we have that $e$ is not minimal, but $\ol{e}$ might be minimal (depending on the labels of the other edges near the same vertex).

\begin{defn}\label{def:flexible-GBS}
    A GBS graph $(\Gamma,\psi)$ is called \emph{flexible} if it satisfies the following three conditions:
    \begin{enumerate}
        \item\label{itm:fully-red} $(\Gamma,\psi)$ is fully reduced.
        \item\label{itm:no-min-min} There is no edge $e\in E(\Gamma)$ such that both $e,\ol{e}$ are minimal.
        \item\label{itm:full-supp-gaps} For every two distinct edges $e,e'\in E(\Gamma)$ with $\tau(e)=\tau(e')$, if $\psi(e)\divides\psi(e')$ then the quotient $\frac{\psi(e')}{\psi(e)}$ is multiple of all the primes in $\cP(\Gamma,\psi)$.
    \end{enumerate}
\end{defn}

The condition of \Cref{itm:full-supp-gaps} is that, whenever two labels near the same vertex divide each other, all primes of $\cP(\Gamma,\psi)$ appear in the factorization of the quotient, each with exponent at least $1$. The reader can think of it as the quotient having ``full support"; this is the kind of condition that we expect to happen with high probability in our random setting.

It is easy to see that, for a GBS graph, we can algorithmically decide whether it is flexible or not. The key feature of flexible GBS graphs is that, using the results of \cite{ACK-iso2}, we can decide the isomorphism problem among them.

\begin{thm}\label{thm:isomorphism-flexible}
There is an algorithm that, given two flexible GBS graphs, determines whether or not they represent isomorphic GBS groups.
\end{thm}
\begin{proof}
    We take a flexible GBS graph $(\Gamma,\psi)$, and we want to prove that it has one qc-class and full-support gaps (see \cite{ACK-iso2} for the definitions). If we do this, then the conclusion follows from the results of \cite{ACK-iso2}.

    First, we show that $(\Gamma,\psi)$ has one qc-class. Let $\ol{e}_0\in E(\Gamma)$ be a minimal edge. By \Cref{itm:no-min-min} of \Cref{def:flexible-GBS} we have that $e_0$ is not minimal, and thus we can find $\ol{e}_1\in E(\Gamma)$ minimal such that $\tau(e_0)=\iota(e_1)$ and $\psi(\ol{e}_1)\divides \psi(e_0)$. Note that the quotient $\psi(e_0)/\psi(\ol{e}_1)$ is multiple of all primes in $\cP(\Gamma,\psi)$, by \Cref{itm:full-supp-gaps} of \Cref{def:flexible-GBS}. Similarly, by induction, we  obtain an infinite path $e_0,e_1,e_2,\dots$ in $\Gamma$ such that $\psi(e_i)/\psi(\ol{e}_{i+1})$ is an integer multiple of all primes in $\cP(\Gamma,\psi)$. But then we must have $e_i=e_j$ for some $i<j$, and this means the path $e_0,e_1,\dots,e_{j-2},e_{j-1},\ol{e}_i,\ol{e}_{i-1},\dots,\ol{e}_1,\ol{e}_0$ conjugates the edge-stabilizer of $e_0$ to a proper power of itself; and this proper power uses all primes in $\cP(\Gamma,\psi)$.

    Using the terminology of \cite{ACK-iso2}, this means that the quasi-conjugacy class containing $e_0$ has quasi-conjugacy support equal to the whole $\cP(\Gamma,\psi)$. Since every edge lies above some minimal edge, we have that every edge is contained in a quasi-conjugacy class with quasi-conjugacy support equal to the whole $\cP(\Gamma,\psi)$. This forces all edges to belong to a common quasi-conjugacy class, and thus $(\Gamma,\psi)$ has one qc-class.

    Now we show that $(\Gamma,\psi)$ has full-support gaps. This follows immediately from the definition of full-support gaps \cite{ACK-iso2} and from \Cref{itm:full-supp-gaps} of \Cref{def:flexible-GBS}.
\end{proof}

\subsection{Rigid GBS graphs}

\begin{defn}\label{def:rigid-GBS}
    A GBS graph $(\Gamma,\psi)$ is called \emph{rigid} if there are no distinct edges $e,e'\in E(\Gamma)$ such that $\tau(e)=\tau(e')$ and $\psi(e)\divides \psi(e')$.
\end{defn}

A rigid GBS admits a unique fully reduced labeled graph representing it (up to signs), and in particular the isomorphism problem is decidable for GBS groups within this class \cite{CF08}.

\section{Random GBS graphs}

In order to model a \emph{random} GBS graph, we shall start by introducing an appropriate notion of a random graph and then prescribe a way to assign random labels to the endpoints of its edges. We intend to fix the number $V$ of vertices and the number $E$ of topological edges (so that, in our previous notation, we will have $|E(\Gamma)|=2E$); later, we shall discuss how to construct the function $\psi$ by selecting $2E$ random labels in such a way as to obtain meaningful results. 

One classical, extremely well-studied model for a finite random graph with a fixed number of vertices $n$ and a fixed number of edges $m$ is the \emph{Erdős–Rényi random graph} $G(n,m)$; this is the graph with vertex set $[n]=\{1,\ldots,n\}$, whose (topological) edges are given by a list of $m$ unordered pairs $\{a_1,b_1\},\ldots,\{a_m,b_m\}$, each pair being chosen uniformly at random within the set ${[n] \choose 2}$ of all possible such pairs, conditionally on the $m$ pairs being all different~\cite{ER60}. Many results are known about the connectivity of such random graphs (as well as several of their features); in particular, there is a phase transition happening around the threshold of $m$ being of order $\frac12n\log n$: for large $n$, the graph goes from being disconnected with high probability (due to the existence of isolated vertices) to being connected with high probability (see~\cite{BBook}).

For our setting, the Erdős–Rényi random graph itself is not well suited, as it is usually built to be simple, whereas our GBS model should allow for loops and multiple edges. We shall therefore make use of the following natural variant, which sometimes goes by \emph{independent edge assignment} (IEA) in the social network literature.

For us, a \emph{random graph} with $V$ vertices and $E$ edges will be a random multigraph on the vertex set $[V]=\{1,\ldots, V\}$ whose (topological) edges are $(\{v_{2i-1},v_{2i}\})_{i=1}^E$, where $(v_i)_{i=1}^{2E}$ are independent, identically distributed random variables, each of them uniformly distributed in $[V]$. Any singletons appearing in the list of edges represent loops.

Although this model is in many ways analogous to the random graph $G(n,m)$ and has similar connectivity properties that could be deduced along the lines of~\cite[Chapter 7]{BBook}, since it is less commonly discussed in the literature, we will give a brief self-contained account of the properties we will use in order to establish the flexibility of a corresponding GBS graph.

\subsection{Connectivity and valences of vertices in a random graph}

The following proposition estimates the probability that a random graph is connected. This is a necessary condition for the random GBS group to be well-defined.

Note that one could achieve even stronger results (it is still true that a phase transition is happening around the threshold of $E\approx \frac12V\log V$, as in simple models); however, we shall later require much stronger lower bounds for $E$, so that we find it convenient to present here a weaker bound with a short proof.

\begin{prop}\label{prop:probability-connected}
Let $\Gamma$ be a random graph with $V$ vertices and $E$ edges, and fix $\delta\in (0,1)$. If $E\geq V(\log V+\log(2/\delta))$, then $\Gamma$ is connected with probability at least $1-\delta$.
\end{prop}

\begin{proof}
Let $S$ be a proper subset of the set of vertices of $\Gamma$ having cardinality $0<s<V$. The probability that there is no edge of $\Gamma$ connecting a vertex in $S$ to a vertex outside of $S$ is $\left(1-\frac{2s(V-s)}{V^2}\right)^E$.
In order to estimate the probability that $\Gamma$ is disconnected, we can just take a union bound over all subsets $S$:
\[\Pr(\Gamma\mbox{ is disconnected})\leq \frac12\sum_{\emptyset\neq S\subsetneq V}\Pr(\mbox{no edge between $S$ and $[V]\setminus S$});\]
this yields
\[\Pr(\Gamma\mbox{ is disconnected})\leq \sum_{s=1}^{\lfloor V/2\rfloor}{V \choose s}\left(1-\frac{2s(V-s)}{V^2}\right)^E\leq \sum_{s=1}^{\lfloor V/2\rfloor}\frac{(V\exp(-E/V))^s}{s!},\]
by using the fact that ${V \choose s}\leq \frac{V^s}{s!}$ and that $s\leq V/2$. By finally extending the sum to all positive integers $s$, under the assumption that $E\geq V(\log V+c)$, we get that the probability of $\Gamma$ being disconnected is at most $\sum_{s=1}^{\infty}\frac{e^{-cs}}{s!}= e^{e^{-c}}-1$. If $c\geq \log(2/\delta)$, then the probability that $\Gamma$ is disconnected is at most $\delta$.
\end{proof}

The following lemma shows that, with high probability, the valence of each vertex is not too far from its expected value of $2E/V$. This is a standard estimate, but for completeness we include a short proof.

\begin{lem}\label{lem:estimate-valences}
Let $\Gamma$ be a random graph with $V$ vertices and $E$ edges. Assuming $E\geq V(\log V+\log(1/\delta))$, the probability that every vertex has valence at least $\frac{1}{6}\frac{2E}{V}$ is at least $1-\delta$.
\end{lem}
\begin{proof}
     We set $a=6$ and $t=\log a$, so that $(1-e^{-t}-\frac{t}{a})\ge\frac{1}{2}$.

    Let $X_1,\dots,X_{2E}$ be independent Bernoulli random variables of parameter $\frac1V$, which take the value $1$ with probability $\frac1V$ and the value $0$ with probability $\frac{V-1}{V}$. For our model of a random (multi)graph, the valence of a fixed vertex has the same distribution as $X=X_1+\dots+X_{2E}$. Thus we obtain the following Chernoff-type bound:
    \[
        \bbP\left(X\le\frac{2E}{aV}\right)
        = \bbP\left(e^{-tX}\ge e^{-t\frac{2E}{aV}}\right)
        \le \frac{\bbE[e^{-tX}]}{e^{-t\frac{2E}{aV}}}
        = e^{\frac{2E}{V}\frac{t}{a}}\left(\bbE[e^{-tX_1}]\right)^{2E} =
    \]
    \[
        = e^{\frac{2E}{V}\frac{t}{a}}\left(1-\frac{1}{V}+\frac{1}{V}e^{-t}\right)^{2E}
        \le e^{\frac{2E}{V}\frac{t}{a}}e^{2E\left(-\frac{1}{V}+\frac{1}{V}e^{-t}\right)}
        = e^{-\frac{2E}{V}(1-e^{-t}-\frac{t}{a})}
        \le e^{-\frac{E}{V}}.
    \]
    Thus the probability that a fixed vertex has valence $\le\frac{2E}{6V}$ is at most $e^{-\frac{E}{V}}$.
    
    Therefore, by a union bound, the probability that some vertex has valence at most $\frac{2E}{6V}$ is at most $Ve^{-\frac{E}{V}}$. Now, assuming $E\geq V(\log V+\log(1/\delta))$, the probability that there is a vertex whose valence is at most $\frac{1}{6}\frac{2E}{V}$ becomes upper bounded by $Ve^{-\frac{E}{V}}\le Ve^{-\log V-\log(1/\delta)}=\delta$.
\end{proof}

\subsection{Flexibility of a random GBS graph}\label{sec:lemmas}
We now introduce the way we intend to sample a random GBS graph throughout this paper.

Having fixed positive integers $V$ (the number of vertices), $E$ (the number of topological edges), as well as a finite set $\mathcal{P}$ of primes and a positive integer $N$, we sample a random GBS graph by taking a random graph $\Gamma$ with $V$ vertices and $E$ edges as in the previous subsection. Independently of the graph $\Gamma$, we consider $2E\cdot D$ i.i.d.~random variables $(\alpha_{i,j})_{i\in[2E], j\in [D]}$, each uniformly distributed in the set $\{0,1,\ldots, N-1\}$, as well as $2E$ independent Bernoulli random variables of parameter $1/2$ (coin flips) $(\epsilon_i)_{i=1}^{2E}$. Letting $\mathcal{P}=\{p_1,\ldots, p_D\}$ and letting $(e_i)_{i=1}^{2E}$ be the elements of $E(\Gamma)$, we attach the label $(-1)^{\epsilon_i}\prod_{j=1}^D p_j^{\alpha_{i,j}}$ to the endpoint $\tau(e_i)$ of the (oriented) edge $e_i$.

As the flexibility of the GBS graph does not depend on the primes $p_1, \ldots, p_D$, we shall simply refer to the random object we just described as \emph{a random GBS graph with parameters $V$, $E$, $D$, $N$}, and always assume that some set of primes of cardinality $D$ has been fixed.

Note that, in order for the GBS group to be well-defined, we need the GBS graph to be connected, whereas this is not necessarily the case in this model. However, it will be convenient to avoid the confusion that comes with building a conditioning into the definition; in any case, in the regimes we are interested in, we have already proved our random underlying graph to be connected with high probability.

In what follows, \Cref{lem:itm-fully-red} and \Cref{lem:itm-full-supp-gaps} give estimates for the probabilities that a random GBS graph satisfies \Cref{itm:fully-red,itm:full-supp-gaps} of \Cref{def:flexible-GBS}. Both bounds are easy to prove under appropriate hypotheses: when the exponents are chosen between $0$ and $N-1$ with $N$ large enough, they will all be distinct and will yield no labels equal to $\pm1$ with high probability.

\begin{lem}\label{lem:itm-fully-red}
    A random GBS graph with parameters $V,E,D,N$ 
    satisfies \Cref{itm:fully-red} of \Cref{def:flexible-GBS} with probability at least $1-\frac{4E}{N^D}$.
\end{lem}
\begin{proof}
    We use \Cref{lem:fully-reduced}: a random label is different from $\pm1$ with probability $1-\frac{1}{N^D}$, and we are choosing $2E$ such labels independently. But we have $(1-\frac{1}{N^D})^{2E}\ge 1-\frac{4E}{N^D}$ and thus we are done.
\end{proof}

\begin{lem}\label{lem:itm-full-supp-gaps}
    Consider a random GBS graph with parameters $V, E, D, N$ such that $N\ge 4E$. The probability that for all $j\in [D]$ the exponents $(\alpha_{i,j})_{i\in[2E]}$ are all distinct is at least $1-\frac{4E^2D}{N}$. In particular, with probability at least $1-\frac{4E^2D}{N}$ the random GBS graph satisfies \Cref{itm:full-supp-gaps} of \Cref{def:flexible-GBS}.
\end{lem}
\begin{proof}
    We are choosing $D$ exponents for each of the $2E$ labels independently and uniformly at random. Each exponent is an integer in $\{0,\dots,N-1\}$. Assuming that $N\geq 4E$, for each $j\in[D]$, the probability that all exponents are distinct is equal to 
    \[
        \left(1-\frac{1}{N}\right)\left(1-\frac{2}{N}\right)\dots\left(1-\frac{2E-1}{N}\right)
        \ge e^{-2\frac{1}{N}}e^{-2\frac{2}{N}}\dots e^{-2\frac{2E-1}{N}} =
    \]
    \[
        = e^{-\frac{2E(2E-1)}{N}}\ge e^{-\frac{4E^2}{N}}
    \]
    where we used that $(1-x)\ge e^{-2x}$ for $0\le x\le 1/2$. 
    Repeating the same reasoning for all $D$ primes, the probability that, for each $j\in[D]$, the exponents $(\alpha_{i,j})_{i\in [2E]}$ are all distinct is at least
    \[
        \left(e^{-\frac{4E^2}{N}}\right)^D\ge 1-\frac{4E^2D}{N}.
    \]
\end{proof}

We now focus on the condition stated in \Cref{itm:no-min-min} of \Cref{def:flexible-GBS}, which is the most substantial. In order to establish whether there exist edges that are minimal in both orientations according to \Cref{def:minimal-edge}, we first need to  estimate how many labels that are minimal for the divisibility relation end up attached to a single endpoint vertex. Note that, given some vertex in a random GBS graph, if we condition it on having valence $k$ (where loops count twice), the $k$ labels it will have attached are of the form $\ell_i=(-1)^{\epsilon_i}\prod_{j=1}^Dp_j^{\alpha_{i,j}}$, for $i\in [k]$, where the random variables $\epsilon_i$ and $\alpha_{i,j}$ are all independent (the former uniform in $\{0,1\}$ and the latter uniform in $\{0,\ldots,N-1\}$). 

Since the signs and the choice of primes are irrelevant for the purpose of establishing divisibility, we may simply consider for each label the vector $(\alpha_{i,j}+1)_{j=1}^D\in [N]^D$, where the shift is just to simplify notation. Given two vectors $X,Y\in \mathbb{R}^D$, we write $X\preceq Y$ if $X(j)\leq Y(j)$ for all $j\in [D]$ and $X\prec Y$ if $X\preceq Y$ and there is $j\in [D]$ such that $X(j)<Y(j)$. This way, the label $\ell_i$ divides the label $\ell_{i'}$ if and only if the corresponding vectors $X_i$, $X_{i'}$ are such that $X_i\preceq X_{i'}$.

Given a set of vectors $S\subseteq \mathbb{R}^D$ and a vector $X\in S$, we shall say that $X$ is \emph{minimal} (in $S$) if there is no point $Y\in S$ such that $Y\prec X$. First of all, we shall sample $2E$ points independently, each uniformly distributed in $[0,1]^D$; we shall index them as $(Y_{i,j})_{i\in [V],j\in[k_i]}$, where the numbers $k_i$ (which will eventually represent the valences of vertices) are positive integers such that $k_1+\ldots+k_V=2E$. Fixing the partition $S_i=\{Y_{i,j}\mid j\in [k_i]\}$, we intend to estimate the probability that the number of points that are minimal in their set $S_i$ is large. Subsequently, from this simple continuous model we shall deduce an estimate for the number of minimal labels in a random GBS graph.

\begin{prop}\label{prop:minimal-continuous-estimate}Having fixed $V, E, D$, let $k_1,\ldots, k_V$ be positive integers such that $k_1+\ldots+k_V=2E$. Fix $\delta\in (0,1)$ and set $L=\log(4E/\delta)$. Let $(Y_{i,j})_{i\in [V],j\in[k_i]}$ be independent random variables, uniformly sampled in $[0,1]^D$, and set $S_i=\{Y_{i,j}\mid j\in [k_i]\}$. For $i\in [V]$, let $M_i=|\{j\in [k_i]\mid Y_{i,j}\mbox{ minimal in }S_i\}|$ and set $M=\sum_{i=1}^V M_i$. Assuming that $\min_{i\in [V]}k_i> 2L$ and that $E\geq e^{2D}$, we have
$$\mathbb{P}\left(M>8LV\frac{(\log E)^{D-1}}{(D-1)!}+2\log(2/\delta)\right)\leq \delta.$$
\end{prop}

\begin{proof}
    Given $y\in[0,1]^D$, set $\rho(y)=y_1\cdot \cdots \cdot y_D$. We shall estimate $M_i$ by partitioning the set of minimal points in $S_i$ into two classes: the set of those minimal points $Y\in S_i$ that have $\rho(Y)\leq 2L/k_i$, of cardinality $A_i$, and the set of all other minimal points in $S_i$, which have $\rho(Y)> 2L/k_i$, of cardinality $B_i$.

    Note that the latter set is likely to be small: given $y\in [0,1]^D$, conditionally on $Y_{i,j}=y$, the probability that $Y_{i,j}$ is minimal in $S_i$ is $(1-\rho(y))^{k_i-1}$. This implies that $\mathbb{E}[B_i]\leq k_i(1-2L/k_i)^{k_i-1}\leq k_ie^{-(k_i-1)2L/k_i}\leq k_ie^{-L}$ (where we have used that $k_i>2L$). Thus, letting $B=\sum_{i=1}^VB_i$, we have $\mathbb{E}[B]\leq 2Ee^{-L}$, which by Markov's inequality implies that $\mathbb{P}(B\geq 1)\leq 2Ee^{-L}=\frac\delta 2$.

    The aim is now to estimate each $A_i$ by just using the fact that, even disregarding minimality, $\rho$ is unlikely to be too small. Indeed, consider a single random variable $Y_{i,j}$, which is uniform in $[0,1]^D$. The probability that $\rho(Y_{i,j})\leq a$ can be computed exactly: it is
    $$\mathbb{P}(\log(\rho(Y_{i,j}))\leq \log(a))=\Pr\left(-\sum_{j=1}^D \log U_j\geq -\log(a)\right),$$
    where the $U_j$'s are independent uniform random variables in $[0,1]$, and therefore the variables $-\log U_j$ are independent exponential random variables of parameter 1. The law of the sum is the Gamma distribution $\Gamma(D,1)$, which has density
    $$f(s)=\frac{s^{D-1}e^{-s}}{(D-1)!}.$$
    It follows that
    $$\Pr(\rho(Y_{i,j})\leq a)=\int_{-\log a}^{\infty}\frac{s^{D-1}e^{-s}}{(D-1)!}\mathrm{d}s=a\sum_{h=0}^{D-1}\frac{\log(1/a)^h}{h!}$$
    where the last well-known identity is obtained by iteratively integrating by parts.
    
    Now, letting $A_i'$ be $|\{j\in [k_i]\mid \rho(Y_{i,j})\leq 2L/k_i\}|$ and setting $A=\sum_{i=1}^VA_i'$, we have 
    $$\mu:=\mathbb{E}[A]=\sum_{i=1}^V \mathbb{E}[A_i']\leq \sum_{i=1}^V k_i\frac{2L}{k_i}\sum_{h=0}^{D-1} \frac{\log(k_i/2L)^h}{h!}\leq 2VL\sum_{h=0}^{D-1} \frac{(\log E)^h}{h!}.$$

In order to simplify the expression above, we may assume as we did in the statement that $\log E\geq 2D> 2D-2$, which would yield
$$\mu\leq 4LV\frac{(\log E)^{D-1}}{(D-1)!}.$$

    Finally, since $A_i'$ is a binomial random variable, we can apply a standard Chernoff bound to the sum of independent Bernoulli random variables $A$ in terms of its expected value $\mu$: 
    we have that $\Pr(A>2\mu+2\log(2/\delta))\leq e^{-\log(2/\delta)}=\delta/2$. Since $M\leq A+B$, this yields 
    $$\Pr(M>2\mu+2\log(2/\delta))\leq \Pr(B>0)+\Pr(B=0, A>2\mu+2\log(2/\delta))\leq\frac\delta2+\frac\delta2=\delta.$$
    By replacing $\mu$ with its upper bound $4LV\frac{(\log E)^{D-1}}{(D-1)!}$, one obtains the expression in the statement.
    
\end{proof}

\begin{cor}\label{cor:minimal-labels-in-GBS-graph}
    Fix positive integer parameters $V, E, D, N$, and $\delta\in(0,1)$, satisfying $E\geq e^{2D}$. Set $L=\log(4E/\delta)$.

    Consider a random GBS graph with parameters $V, E, D, N$ conditioned on each vertex having valence more than $2L$ and on no two labels having the same exponent for any prime $p\in \mathcal{P}$. Let $M$ be the number of labels among the $2E$ sampled ones that are minimal for the divisibility relation among labels attached to the same endpoint vertex. As in \Cref{prop:minimal-continuous-estimate}, we have
    $$\Pr\left(M\leq 8LV\frac{(\log E)^{D-1}}{(D-1)!}+2\log(2/\delta)\right)\geq 1-\delta.$$
\end{cor}

\begin{proof}
Conditionally on the random graph $\Gamma$, assuming the valence of vertex $i$ is $k_i$, consider the random labels $(\ell_{i,j})_{i\in [V],j\in [k_i]}$. For each of these, one can consider the corresponding vector $X_{i,j}\in [N]^D$ defined as described at the beginning of this section, whose entries are the $D$ exponents augmented by one. If one ignores the conditioning, the vectors $(X_{i,j})_{i\in [V],j\in [k_i]}$ are independent and uniformly distributed in $[N]^D$. Under the conditioning that there are no pairs $(i,j)$ and $(i',j')$ and no $d\in[D]$ such that $X_{i,j}(d)=X_{i',j'}(d)$, for each $d\in [D]$ the ordering of the values $X_{i,j}(d)$ induces a uniform random permutation $\sigma_d$ of the set $\{(i,j)\mid i\in [V],j\in [k_i]\}$. The $D$ permutations are all independent and fully determine the original divisibility relation $\prec$. If one were to instead sample independent random variables $Y_{i,j}$, each uniformly distributed in $[0,1]^D$, those would also induce uniform permutations, independent for each $d\in [D]$. It follows that the random variable $M$ from \Cref{prop:minimal-continuous-estimate} has the same distribution as the number of minimal labels under this conditioning.
\end{proof}

\begin{remark}
    Note that one could compute the actual expected value of the total number of minimal labels. Assuming valences $k_1,\ldots,k_V$ and under the conditioning that no two sampled exponents are the same, the probability that a given label attached to vertex $i$ is minimal is
    $$\int_{[0,1]^D}(1-\rho(y_1,\ldots,y_D))^{k_i-1}\mathrm{d}y_1\cdot\cdots\cdot\mathrm{d}y_D=\sum_{h=0}^{k_i-1}(-1)^h{k_i-1 \choose h}\frac{1}{(h+1)^D}.$$
    If $D$ is fixed and each of the $k_i$'s is large, the asymptotics for this probability are of order $\frac{(\log k_i)^{D-1}}{k_i(D-1)!}$; by concavity of the logarithm, this would yield an upper bound for the expected total number of minimal points of order (up to a multiplicative constant) $V\frac{\log(2E/V)^{D-1}}{(D-1)!}$.

    Obtaining an actual tight second moment bound from here does not particularly further our aims within this paper, so we chose to give a fairly elementary argument for the estimate of \Cref{prop:minimal-continuous-estimate}. Note that, though the order of our estimate is indeed worse, it only differs by logarithmic factors in $E$ with respect to a sharp one.
\end{remark}

Finally, in order to estimate the probability that there is an edge with two minimal labels attached to its endpoints, we use the following elementary lemma:

\begin{lem}\label{lem:no-pair-of-minimals}
    Given a list of $2E$ objects, of which $M\leq E$ are \emph{marked} and $2E-M$ \emph{unmarked}, we pair them up at random. The probability that no marked object is paired with another marked object is at least $1-\frac{M^2}{2E}$.
\end{lem}
\begin{proof}Fix a pair of objects $\{x,y\}$; the probability that $x$ and $y$ are matched in the random pairing is $\frac{1}{2E-1}$ (as $x$ is matched with a uniform object other than itself). Consider now the number of pairs consisting of two marked objects: by linearity of expectation, a uniform pairing will have on average $\frac{1}{2E-1}{M \choose 2}$. By Markov's inequality, the probability that there is at least one such pair is thus at most $\frac{1}{2E-1}{M \choose 2}\leq \frac{M^2}{2E}$. 
\end{proof}

\subsection{Random GBS graphs with many edges}

Throughout this section, we consider a random GBS with $V$ vertices, $E$ edges, $D$ primes and exponents up to $N$. By combining the results of the previous section, we estimate the probability that a random GBS is flexible (see \Cref{def:flexible-GBS}).

\begin{thm}\label{thm:random-GBS}
    Let $0<\epsilon<1/50$. Consider a random GBS graph with parameters $V, E, D, N$. 
    Suppose that $E\ge V^{2+\epsilon}$,  $E\ge e^{2D/\epsilon^2}$ and $N\ge E^4$. Then the GBS graph is connected and flexible with probability at least $1-\frac{1}{E^{\epsilon/8}}$.
\end{thm}
\begin{proof}

Let $\Gamma$ be a random GBS graph with $V$ vertices and $E$ edges, and let $(\ell_e)_{e\in [2E]}$ be the random labels attached to the oriented edges, which we label as $1,\ldots,2E$.

We have already seen that a number of events jointly occur with high probability: the random graph $\Gamma$ is connected, the labels satisfy both \Cref{itm:fully-red} of \Cref{def:flexible-GBS} and \Cref{itm:full-supp-gaps} of \Cref{def:flexible-GBS}.

Indeed, by \Cref{prop:probability-connected} used with $\delta=\frac{1}{E}$, the graph is connected with probability at least $1-\frac{1}{E}$: we have $E\geq V^{2+\epsilon}$ and $E\geq e^{2D/\epsilon^2}\geq e^{2\cdot 50^2}$, so $\frac{E}{2}\geq V\log V$ and $\frac{E}{2}\geq V\log (2E)$. 

    By \Cref{lem:itm-fully-red}, the GBS graph satisfies \Cref{itm:fully-red} of \Cref{def:flexible-GBS} with probability at least $1-\frac{4E}{N^D}\ge 1-\frac{1}{E}$.

    By \Cref{lem:itm-full-supp-gaps}, the sampled labels have $2E$ pairwise distinct exponents for each of the $D$ primes (and thus the random GBS graph satisfies \Cref{itm:full-supp-gaps} of \Cref{def:flexible-GBS}) with probability at least $1-\frac{4E^2D}{N}\ge 1-\frac{1}{E}$.

    We now focus on \Cref{itm:no-min-min} of \Cref{def:flexible-GBS}. In \Cref{cor:minimal-labels-in-GBS-graph}, take again $\delta=1/E$ so that $L=2\log(2E)$. By \Cref{lem:estimate-valences}, every vertex of a random graph $\Gamma$ with $V$ vertices and $E$ edges has valence at least $\frac{1}{6}\frac{2E}{V}>4\log(2E)=2L$ with probability at least $1-\frac{1}{E}$.

    With probability at least $1-\frac{1}{E}$, under the valence and no-ties-between-exponents conditionings, the total number of minimal labels is at most
    $$8LV\frac{(\log E)^{D-1}}{(D-1)!}+2\log(2E).$$

    We now use the fact that $V\leq E^{\frac{1}{2+\epsilon}}$ and $D\leq \frac{\epsilon^2}{2}\log E$, as well as $2\log(2E)\leq E^{\frac{1}{2+\epsilon}}$: we have
    
    $$8L\frac{(\log E)^{D-1}}{(D-1)!}+1\leq 16\log(2E)\frac{(\log E)^{D-1}}{(D-1)!}.$$

    For $D=1$, the above is at most $16\log(2E)\leq E^{\frac{\epsilon}{6}}$. For $D\geq 2$, by a Stirling estimate, we have $(D-1)!\geq (\frac{D-1}{e})^{D-1}$, and since $(ex/k)^k$ is increasing in $k$ for $k<x$, we can upper bound the expression above by
    $$16\log(2E)\left(\frac{e\log E}{D}\right)^D\leq16\log(2E)\left(\frac{2e}{\epsilon^2}\right)^{\frac{\epsilon^2}{2} \log E}\leq 16\log(2E)E^{\frac{\epsilon^2}{2}(\log(2/\epsilon^2)+1)}\leq E^{\frac{\epsilon}{6}}.$$
    
    Therefore, the bound for the number of minimal labels is at most $E^{\frac{1}{2+\epsilon}+\frac{\epsilon}{6}}\leq E^{\frac12-\frac \epsilon {16}}$. 
    
    It is now useful to remark that sampling a random graph with $V$ vertices and $E$ edges, conditioned on having valences $k_1,\ldots,k_V$, can also be done by making a list of vertices $(v_i)_{i\in [2E]}$ with $k_i$ copies of vertex $i$ for each $i\in [V]$ and sampling (topological) edges by picking a uniform random pairing of the elements of $[2E]$, adding the edge $\{v_i,v_j\}$ (or a loop around $v_i$ if $i=j$) if $i$ and $j$ are paired.
    
    This way, we obtain a lower bound of $1-\frac{E^{1-\frac\epsilon {8}}}{2E}=1-\frac{1}{2E^{\epsilon/8}}$ for the probability that the GBS graph is flexible, by \Cref{lem:no-pair-of-minimals}.

    Putting everything together, the probability that the GBS graph is not flexible or not connected is at most, by a straightforward union bound, the sum of the probabilities of several bad events, plus the conditional probability that two minimal labels are matched. We obtain an upper bound of $\frac{5}{E}+\frac{1}{2E^{\epsilon/8}}\leq\frac{1}{E^{\epsilon/8}}$, hence the estimate in the theorem statement.
\end{proof}

\begin{thm}\label{thm:random-GBS-one-vertex}
    Consider a random GBS graph with $V=1$ vertex, $E$ edges, $D$ primes and exponents up to $N$. Suppose that $E$ is large enough and $E\ge e^{9D}$ and $N\ge E^4$. Then the GBS is flexible with probability at least $1-\frac{1}{\sqrt[4]{E}}$.
\end{thm}
\begin{proof}
Note that a random graph $\Gamma$ with one vertex is automatically connected.

By \Cref{lem:itm-fully-red}, the random GBS graph satisfies \Cref{itm:fully-red} of \Cref{def:flexible-GBS} with probability $\ge 1-\frac{4E}{N^D}\ge 1-\frac{1}{E}$.

    By \Cref{lem:itm-full-supp-gaps}, the random GBS graph satisfies \Cref{itm:full-supp-gaps} of \Cref{def:flexible-GBS} with probability $\ge 1-\frac{4E^2D}{N}\ge 1-\frac{1}{E}$.

    We can now apply \Cref{cor:minimal-labels-in-GBS-graph} with $\delta=1/E$ and thus $L=2\log(2E)$: the number of minimal labels is at most
    $$8L\frac{\left(\log E\right)^{D-1}}{(D-1)!}+2\log(2E)$$
    with probability at least $1-\frac1 E$. On the other hand, since $D\leq\frac{\log E}{9}$, by a Stirling estimate similar to the one in the previous proof we obtain that the bound above is at most
    $$16\log(2E)(9e)^{\frac19\log E}\leq E^{3/8}$$
    when $E$ is large. Finally, by \Cref{lem:no-pair-of-minimals}, having performed all the previous conditionings, the probability that two minimal labels end up attached to opposite ends of the same edge is at most $\frac{(E^{3/8})^2}{2E}=\frac{1}{2 E^{1/4}}$.

Our random GBS graph will therefore be flexible with probability at least $1-\frac{3}{E}-\frac{1}{2E^{1/4}}\geq 1-\frac{1}{E^{1/4}}$ (for $E$ large enough).
\end{proof}

\section{An alternative random model}

In this section we consider an alternative definition of a random GBS with parameters $V,E,L$: we first sample a random graph with $V$ vertices and $E$ edges, then choose the $2E$ labels independently and uniformly at random in the set $\{\pm1,\pm2,\dots,\pm L\}$.

\begin{lem}\label{lem:all-minimal}
    Suppose that $k\le\sqrt[4]{L}$ and $L\ge 2^{16}$. Choose $k$ integers independently and uniformly at random in $[1,L]$. The probability that one of them divides another is at most $\frac{1}{\sqrt[4]{L}}$.
\end{lem}
\begin{proof}
Let $a_1,\ldots,a_k$ be the chosen integers. For two distinct indices $i\neq j$, we have
$$\Pr(a_i\mid a_j)=\frac{1}{L^2}\sum_{m=1}^L\lfloor \frac{L}{m}\rfloor\leq\frac{1}{L}\sum_{m=1}^L \frac{1}{m}\leq\frac{1+\log L}{L}.$$

By taking a union bound, the probability that some $a_i$ divides some $a_j$ with $j\neq i$ is at most $k(k-1)\frac{1+\log L}{L}\leq \frac{1+\log L}{\sqrt{L}}$, since $k\leq \sqrt[4]{L}$. For $L\geq 2^{16}$, the bound is at most $\frac{1}{\sqrt[4]{L}}$.
\end{proof}

\begin{thm}\label{thm:random-GBS-is-rigid}
    Consider a random GBS with parameters $V,E,L$, conditioned on the graph being connected. Suppose that $L\ge (2E)^4$. Then the GBS is rigid with probability at least $1-\frac{1}{\sqrt[4]{L}}$.
\end{thm}
\begin{proof}
    We are choosing $2E$ non-zero integers, independently and uniformly at random in $\{\pm 1,\pm 2, \dots, \pm L\}$. By \Cref{lem:all-minimal}, with probability at least $1-\frac{1}{\sqrt[4]{L}}$ there is no divisibility relation among them. In particular, there is no divisibility relation between two labels at the same vertex, and thus the GBS is rigid.
\end{proof}

\bibliographystyle{alpha}
{\footnotesize\bibliography{bibliography.bib}}

\end{document}